%% file: main.tex
\documentclass[12pt]{article}
\usepackage{fullpage}
\usepackage[utf8]{inputenc}
\usepackage{amsmath, amssymb, amsfonts, amsthm, mathtools, xcolor, stackengine}
\usepackage{tikz}
\usepackage{scrextend}
\usepackage{hyperref,url}
\usepackage{changepage}

\newcommand{\setbuilder}[2]{\left\{#1\ \colon #2\right\}}

\DeclareMathOperator{\conv}{conv}

\newcommand{\R}{{\mathbb R}}
\newcommand{\N}{{\mathbb N}}
\newcommand{\Z}{{\mathbb Z}}
\newcommand{\Q}{{\mathbb Q}}
\renewcommand{\L}{{\mathcal L}}

\newlength{\NOTskip}

\theoremstyle{plain}
\newtheorem{theorem}{Theorem}
\newtheorem{claim}[theorem]{Claim}
\newtheorem{problem}[theorem]{Problem}
\newtheorem{corollary}[theorem]{Corollary}
\newtheorem{proposition}[theorem]{Proposition}
\newtheorem{lemma}[theorem]{Lemma}
\newtheorem{prop}[theorem]{Proposition}
\newtheorem{conjecture}[theorem]{Conjecture}
\newtheorem{question}[theorem]{Question}

\theoremstyle{definition}

\newtheorem{remark}{Remark}
\newtheorem*{counterexample}{Counterexample}

\date{}

\author{N\'ora Frankl\footnote{Carnegie Mellon University, Pittsburgh, and Geometric Structures at MIPT, Moscow. Research was part of project $\textrm{2018-2.1.1-UK}\_\textrm{GYAK-2018-00024}$ of the summer internship for Hungarian students studying in the UK, supported by the National Research, Development and Innovation Office. NF also acknowledges the financial support from the Ministry of Education and Science of the Russian Federation in the framework of MegaGrant no 075-15-2019-1926. Research was also partially supported by the National Research, Development, and Innovation Office, NKFIH Grant K119670 and by ERC Advanced Grant "GeoScape."} \and Tam\'as Hubai\footnote{MTA-ELTE Lend\"ulet Combinatorial Geometry Research Group, Institute of Mathematics, E\"otv\"os Lor\'and University (ELTE), Budapest, Hungary. Research supported by the Lend\"ulet program of the Hungarian Academy of Sciences (MTA), under grant number LP2017-19/2017.}\and  D\"om\"ot\"or P\'alv\"olgyi$^\dagger$}

\title{Almost-monochromatic sets
and the chromatic number of the plane}

\begin{document}
\maketitle

\begin{abstract}
    In a colouring of $\mathbb{R}^d$ a pair $(S,s_0)$ with $S\subseteq \mathbb{R}^d$ and with $s_0\in S$ is \emph{almost-monochromatic} if $S\setminus \{s_0\}$ is monochromatic but $S$ is not. We consider questions about finding almost-monochromatic similar copies of pairs $(S,s_0)$ in colourings of $\mathbb{R}^d$, $\mathbb{Z}^d$, and of $\mathbb{Q}$ under some restrictions on the colouring.
    
    Among other results, we characterise those $(S,s_0)$ with $S\subseteq \mathbb{Z}$ for which every finite colouring of $\mathbb{R}$ without an infinite monochromatic arithmetic progression contains an almost-monochromatic similar copy of $(S,s_0)$. We also show that if $S\subseteq \mathbb{Z}^d$ and $s_0$ is outside of the convex hull of $S\setminus \{s_0\}$, then every finite colouring of $\mathbb{R}^d$ without a monochromatic similar copy of $\mathbb{Z}^d$ contains an almost-monochromatic similar copy of $(S,s_0)$.
    Further, we propose an approach based on finding almost-monochromatic sets that might lead to a human-verifiable proof of $\chi(\R^2)\geq 5$.
\end{abstract}

\section{Introduction}

A colouring $\varphi: \R^2\to \{1,\dots,k\}$ is a (unit-distance-avoiding) \emph{proper $k$-colouring} of the plane, if $\|p-q\|=1$ implies $\varphi(p)\ne \varphi(q)$, where $\|.\|$ denotes the Euclidean norm. The \emph{chromatic number $\chi(\mathbb{R}^2)$ of the plane} is the smallest $k$ for which there exists a proper $k$-colouring of the plane. Determining the exact value of $\chi(\mathbb{R}^2)$, also known as the Hadwiger-Nelson problem, is a difficult problem. 
In $2018$ Aubrey de Grey~\cite{degrey} showed that $\chi(\mathbb{R}^2)\geq 5$, improving the long standing previous lower bound $\chi(\mathbb{R}^2)\geq 4$ which was first noted by Nelson (see \cite{soiferbook}). The best known upper bound $\chi(\mathbb{R}^2)\leq 7$ was first observed by Isbell (see \cite{soiferbook}), and it is widely conjectured that $\chi(\mathbb{R}^2)=7$.
For history and related results we refer the reader to Soifer's book \cite{soiferbook}.

A graph $G=(V,E)$ is a \emph{unit-distance graph} in the plane if $V\subseteq \mathbb{R}^2$ such that if $(v,w)\in E$ then $\|v-w\|=1$. De Grey constructed a unit-distance graph $G$ with $1581$ vertices, and checked that $\chi(G)\geq 5$ by a computer program. Following his breakthrough, a polymath project, Polymath$16$ \cite{polymath} was launched with the main goal of finding a human-verifiable proof of $\chi(\mathbb{R}^2)\geq 5$. Following ideas proposed in Polymath$16$ by the third author \cite{domotor}, we present an approach that might lead to a human-verifiable proof of $\chi(\mathbb{R}^2)\geq 5$.

We call a collection of unit circles $C=C_1\cup\dots\cup C_n$ having a common point $O$ a \emph{bouquet through $O$}. For a given colouring of $\mathbb{R}^2$, the bouquet $C$ is \emph{bold} if there is a colour, say blue, such that every circle $C_i$ has a blue point, but $O$ is not blue.

\begin{conjecture}\label{introbouquet} For every bouquet $C$, every colouring of the plane with finitely many but at least two colours contains a bold congruent copy of $C$.
\end{conjecture}

In Section~\ref{4} we show that the statement of Conjecture~\ref{introbouquet} would provide a human-verifiable proof of $\chi(\mathbb{R}^2)\geq 5$. 
We prove the conjecture for a specific family of bouquets with proper colourings of $\mathbb{R}^2$.

\begin{theorem}\label{introracbouq}
Let $C=C_1\cup\dots\cup C_n$ be a bouquet through $O$ and for every $i$ let $O_i$ be the centre of $C_i$. If $O$ and $O_1,\dots,O_n$ are contained in $\mathbb{Q}^2$, further $O$ is an extreme point of $\{O,O_1,\dots,O_n\}$, then Conjecture~\ref{introbouquet} is true for $C$ for every proper colouring of $\mathbb{R}^2$.
\end{theorem}

In Section \ref{42} we prove a more general statement which implies Theorem \ref{introracbouq}. We also prove a statement similar to that of Conjecture~\ref{introbouquet} for concurrent lines. We call a collection of lines $L=L_1\cup\dots\cup L_n$ with a common point $O$ a \emph{pencil through $O$}. The pencil $L$ is \emph{bold} if there is a colour, say blue, such that every line $L_i$ has a blue point, but $O$ is not blue.

\begin{theorem}\label{intropencil} For every pencil $L$, every colouring of the plane with finitely many but at least two colours contains a bold congruent copy of $L$.
\end{theorem}

\subsection{Almost-monochromatic sets}

Let $S\subseteq \mathbb{R}^d$ be a finite set with $|S|\geq 3$, and let $s_0\in S$. In a colouring of $\mathbb{R}^d$ we call $S$ \emph{monochromatic}, if every point of $S$ has the same colour. A pair $(S,s_0)$ is \emph{almost-monochromatic} if $S\setminus \{s_0\}$ is monochromatic but $S$ is not. 
From now on we will use the abbreviation AM for almost-monochromatic.

We call a colouring a \emph{finite colouring}, if it uses finitely many colours. An \emph{infinite arithmetic progression} in $\mathbb{R}^d$ is a similar copy of $\mathbb{N}$.
From now on we will use the abbreviation AP for infinite arithmetic progression.
A colouring is \emph{AP-free} if it does not contain a monochromatic infinite arithmetic progression.

Motivated by its connections to the chromatic number of the plane,\footnote{The connection is described in details later; see Theorems~\ref{human}, \ref{latticeboq} and \ref{implication}.} we propose to study the following problem.

\begin{problem}\label{problem1} Characterise those pairs $(S,s_0)$ with $S\subseteq \mathbb{R}^d$ and with $s_0\in S$ for which it is true that every AP-free finite colouring of $\mathbb{R}^d$ contains an AM similar copy of $(S,s_0)$.
\end{problem}

Note that finding an AM \emph{congruent} copy of a given pair $(S,s_0)$ was studied by Erd\H os, Graham, Montgomery, Rothschild, Spencer, and Strauss \cite{EGMRSS3}. We solve Problem~\ref{problem1} in the case when $S\subseteq \mathbb{Z}^d$. A point $s_0\in S$ is called an \emph{extreme point} of $S$ if $s_0\notin \conv (S\setminus \{s_0\})$. 

\begin{theorem}\label{introZ}
Let $S\subseteq \mathbb{Z}^d$ and $s_0\in S$. Then there is an AP-free colouring of $\mathbb{R}^d$ without an AM similar copy of $(S,s_0)$ if and only if $|S|>3$ and $s_0$ is not an extreme point of $S$.
\end{theorem}

We prove Theorem~\ref{introZ} in full generality in Section \ref{32}. The ‘only if’ direction will follow from a stronger statement, Theorem~\ref{grid}. In Section \ref{21} we consider only $d=1$, the $1$-dimensional case. We prove some statements similar to Theorem~\ref{introZ} for $d=1$, and illustrate the ideas that are later used to prove the theorem in general.

Problem~\ref{problem1} is related to and motivated by Euclidean Ramsey theory, a topic introduced by Erd\H os, Graham, Montgomery, Rothschild, Spencer, and Strauss \cite{EGMRSS}. Its central question asks to find those finite sets $S\subseteq \mathbb{R}^d$ for which the following is true. \emph{For every $k$ if $d$ is sufficiently large, then every colouring of $\mathbb{R}^d$ using at most $k$ colours contains a monochromatic congruent copy of $S$.}
Characterising sets having the property described above is a well-studied difficult question, and is in general wide open. For a comprehensive overview see Graham's survey~\cite{graham}.

The nature of the problem significantly changes if instead of a monochromatic congruent copy we ask for a monochromatic similar copy, or a monochromatic homothetic copy.
A \emph{(positive) homothetic copy} (or \emph{(positive) homothet}) of a set $H\subseteq \mathbb{R}^d$ is a set $c+\lambda H=\setbuilder{c+\lambda h}{h\in H}$ for some $c\in\mathbb{R}^d$ and some (positive) $\lambda\in \mathbb{R}\setminus\{0\}$. Gallai proved that if $S\subseteq \mathbb{R}^d$ is a finite set, then every colouring of $\mathbb{R}^d$ using finitely many colours contains a monochromatic positive homothetic copy of $S$. This statement first appeared in the mentioned form in the book of Graham, Rothschild, and Spencer~\cite{gallaiR}.

A direct analogue of Gallai's theorem for AM sets is not true: there is no AM similar copy of any $(S,s_0)$ if the whole space is coloured with one colour only. However, there are pairs $(S,s_0)$ for which a direct analogue of Gallai's theorem is true for colourings of $\mathbb{Q}$ with more than one colour. In particular, we prove the following result in Section \ref{inQ}.

\begin{theorem}\label{iQ} Let $S=\{0,1,2\}$ and $s_0=0$. Then every finite colouring of $\mathbb{Q}$ with more than one colour contains an AM positive homothet of $(S,s_0)$.
\end{theorem}

In general, we could ask whether every non-monochromatic colouring of $\mathbb{R}^d$ with finitely many colours contains an AM similar copy of every $(S,s_0)$. This, however, is false, as shown by the following example from \cite{EGMRSS3}. Let $S=\{1,2,3\}$ and $s_0=2$. If $\mathbb{R}_{>0}$ is coloured red and $\mathbb{R}_{\leq 0}$ is coloured blue, we obtain a colouring of $\mathbb{R}$ without an AM similar copy of $(S,s_0)$. 
Restricting the colouring to $\mathbb N$, using the set of colours $\{0,1,2\}$ and colouring every $n\in \mathbb{N}$ with $n$ modulo $3$, we obtain a colouring without an AM similar copy of $(S,s_0)$.
However, notice that in both examples each colour class contains an infinite monochromatic AP.

Therefore, our reason, apart from its connections to the Hadwiger-Nelson problem, for finding AM similar copies of $(S,s_0)$ in AP-free colourings was to impose a meaningful condition to exclude `trivial' colourings.

\section{The line}\label{21}
In this section we prove a statement slightly weaker than Theorem~\ref{introZ} for $d=1$. The main goal of this section to illustrate some of the ideas that we use to prove Theorem \ref{introZ}, but in a simpler case. Note that in $\mathbb{R}$ the notion of similar copy and homothetic copy is the same. 

\begin{theorem}\label{introsimilar1}
Let $S\subseteq \mathbb{Z}$ and $s_0\in S$. Then there is an AP-free colouring of $\mathbb{N}$ and of $\mathbb{R}$ without an AM positive homothetic copy of $(S,s_0)$ if and only if $|S|>3$ and $s_0$ is not an extreme point of $S$.
\end{theorem}

To prove Theorem \ref{introsimilar1} it is sufficient to prove the `if' direction only for $\mathbb{R}$ and the `only if' direction only for $\mathbb{N}$. Thus it follows from the three lemmas below, that consider cases of Theorem \ref{introsimilar1} depending on the cardinality of $S$ and on the position of $s_0$.

\begin{lemma}\label{1extreme} If $s_0$ is an extreme point of $S$, then every finite AP-free colouring of $\mathbb{N}$ contains an AM positive homothetic copy of $(S,s_0)$.
\end{lemma}

\begin{lemma}\label{three} If $|S|=3$, then every AP-free finite colouring of $\mathbb{N}$ contains an AM positive homothetic copy of $(S,s_0)$.
\end{lemma}

\begin{lemma}\label{1notextreme} If $S\subseteq \mathbb{R}$, $|S|>3$ and $s_0$ is not an extreme point of $S$, then there is an AP-free finite colouring of $\mathbb R$ without an AM positive homothetic copy of $(S,s_0)$.
\end{lemma}

Before turning to the proofs, recall Van der Waerden's theorem \cite{vdW} and a corollary of it. A colouring is a \emph{$k$-colouring} if it uses at most $k$ colours.

\begin{theorem}[Van der Waerden \cite{vdW}]\label{vdWoriginal}
For every $k,\ell\in \mathbb{N}$ there is an $N(k,\ell)\in \mathbb{N}$ such that every $k$-colouring of $\{1,\ldots,N(k,\ell)\}$ contains an $\ell$-term monochromatic AP.
\end{theorem}

\begin{corollary}[Van der Waerden \cite{vdW}]\label{vdW}
For every $k,\ell\in \mathbb{N}$ and for every $k$-colouring of $\N$ there is a $t\le N(k,\ell)$ such that there are infinitely many monochromatic $\ell$-term AP of the same colour with difference $t$. 
\end{corollary}

\begin{proof}[Proof of Lemma~\ref{1extreme}]
Let $S=\{p_1,\dots,p_n\}$ with $1< p_1< \dots < p_n$ and $\varphi$ be an AP-free colouring of $\mathbb{N}$.
If $s_0$ is an extreme point of $S$, then either $s_0=p_1$ or $s_0=p_n$.

\medskip

\textbf{Case 1:} $s_0=p_n$. By Theorem~\ref{vdWoriginal} $\varphi$ contains a monochromatic positive homothet \mbox{$M+\lambda([1,p_n)\cap \N)$} of $[1,p_n)\cap \N$ of colour, say, blue. Observe that since $\varphi$ is AP-free there is a \mbox{$q\in M+\lambda ([p_n,\infty)\cap \mathbb{N})$} which is not blue. Let $M+q\lambda$ be the smallest non-blue element in $M+\lambda([p_n,\infty)\cap \mathbb{N})$. Then $(M+\lambda(q-p_n)+\lambda S,M+\lambda q)$ is an AM homothet of $(S,s_0)$.
	
	\medskip

\textbf{Case 2:} $s_0=p_1$. By Corollary~\ref{vdW}  there is a $\lambda \in \mathbb{N}$ such that $\varphi$ contains infinitely many monochromatic congruent copies of $\lambda ((1,p_n]\cap \N)$, say of colour blue.
Without loss of generality, we may assume that infinitely many of these monochromatic copies are contained in $\lambda \N$.
Since $\varphi$ is AP-free, $\lambda \N$ is not monochromatic,  and thus there is an $i$ such that $i\lambda$ and $(i+1)\lambda$ are of different colours. 
Consider a blue interval $M+\lambda ((1,p_n]\cap \N)$  such that $M+\lambda>i\lambda$, and let $q$ be the largest non-blue element of $[1,M+\lambda)\cap \lambda \N$. This largest element exists since $\lambda i$ and $\lambda (i+1)$ are of different colour. Then $(q-\lambda p_1+\lambda S,q)$ is an AM homothet of $(S,s_0)$. 
\end{proof}

\begin{proof}[Proof of Lemma~\ref{three}]
Let $S=\{p_1,p_2,p_3\}$ with $1< p_1<p_2 < p_3$ and $\varphi$ be an AP-free colouring of $\mathbb{N}$.
We may assume that $s_0=p_2$, otherwise we are done by Lemma~\ref{1extreme}. 

There is an $r\in \mathbb{Q}_{>0}$ such that $\{q_1,q_2,q_3\}$ is a positive homothet of $S$ if and only if $q_2=r q_1+(1-r)q_3$. Fix an $M\in \mathbb{N}$ for which $Mr\in \N$.
We say that $I$ is an interval of $c+\lambda \mathbb{N}$ of length $\ell$ if there is an interval $J\subseteq \mathbb{R}$ such that $I=J\cap (c+\lambda \mathbb{N})$ and $|I|=\ell$.

\begin{prop}\label{propinterval} Let $I_1$ and $I_3$ be intervals of $\lambda \mathbb{N}$ of length $2M$ and $M$ respectively such that $\max I_1 <\min I_3$. Then there is an interval $I_2\subseteq \lambda\mathbb{N}$ of length $M$ such that $\max I_1 <\max I_2<\max I_3$, and for every $q_2\in I_2$ there are $q_1\in I_1$ and $q_3\in I_3$ such that $\{q_1,q_2,q_3\}$ is a positive homothetic copy of $S$. 
    \end{prop}
    
\begin{proof} Without loss of generality we may assume that $\lambda=1$. Let $I_1^L$ be the set of the $M$ smallest elements of $I_1$.
By the choice of $M$ for any $q_3\in \mathbb N$ the interval $r I_1^L+(1-r)q_3$ contains at least one natural number.
Let $q_3$ be the smallest element of $I_3$ and $q_1\in I_1^L$ such that $r q_1+(1-r)q_3\in \N$.
Then $I_2=\setbuilder{r (q_1+i)+(1-r)(q_3+i)}{0\leq i<M}$ is an interval of $\mathbb{N}$ of length $M$ satisfying the requirements, since $q_1+i\in I_1$ and $q_3+i\in I_3$.
\end{proof}
    
We now return to the proof of Lemma~\ref{three}.
Let $I$ be an interval of $\N$ of length $2M$. By Corollary~\ref{vdW} there is a $\lambda\in \N$ such that $\varphi$ contains infinitely many monochromatic copies of $\lambda I$ of the same colour, say of blue. Moreover, by the pigeonhole principle there is a $c\in \mathbb{N}$ such that infinitely many of these blue copies are contained in $c+\lambda\mathbb{N}$, and without loss of generality we may assume that $c=0$. 
    
Consider a blue interval $[a\lambda,a\lambda+2M\lambda-\lambda]$ of $\lambda \mathbb{N}$ of length $2M$.
Since $\varphi$ is AP-free, $[a\lambda+2M\lambda,\infty)\cap\lambda \mathbb{N}$ is not completely blue.
Let $q\lambda$ be its smallest element  which is not blue and let $I_1=[q\lambda-2M\lambda,(q-1)\lambda]\cap\lambda \mathbb{N}$. Let $I_3$ be the blue interval of length $M$ in $\lambda \mathbb{N}$ with the smallest possible $ \min I_3$ for which $\max I_1 <\min I_3$. Then Proposition~\ref{propinterval} provides an AM positive homothet of $(S,s_0)$. 
    
Indeed, consider the interval $I_2$ given by the proposition. There exists a $q_2\in I_2$ which is not blue, otherwise every point of $I_2$ is blue, contradicting the minimality of $\min I_3$. But then there are $q_1\in I_1$, $q_3\in I_3$ such that $(\{q_1,q_2,q_3\},q_2)$ is an AM homothet copy of $(S,s_0)$.
\end{proof}

\begin{proof}[Proof of Lemma \ref{1notextreme}]
$S$ contains a set $S'$ of $4$ points with $s_0\in S'$ such that $s_0$ is not an extreme point of $S'$. Thus we may assume that $S=\{p_1,p_2,p_3,p_4\}$ with $p_1<p_2<p_3<p_4$ and that $s_0=p_2$ or $s_0=p_3$.
We construct the colouring for these two cases separately.
First we construct a colouring $\varphi_1$ of $\mathbb{R}_{ >0}$ for the case of $s_0=p_3$, and a colouring $\varphi_2$ of $\mathbb{R}_{\geq 0}$ for the case of $s_0=p_2$. Then we extend the colouring in both cases to $\mathbb{R}$.

\medskip

\textbf{Construction of $\bf {\varphi_1}$ ($\bf s_0=p_3$) :} \rm Fix $K$ such that $K>\frac{p_4-p_2}{p_2-p_1}+1$ and let $\{0,1,2\}$ be the set of colours. We define $\varphi_1$ as follows. Colour $(0,1)$ with colour $2$, and for every $i\in \N\cup \{0\}$ colour $[K^i,K^{i+1})$ with $i$ modulo $2$. The colouring $\varphi_1$ defined this way is AP-free, since it contains arbitrarily long monochromatic intervals of colours $1$ and $2$. Thus we only have to show that it does not contain an AM positive homothet of $(S,s_0)$.

Consider a positive homothet $c+\lambda S=\{r_1,r_2,r_3,r_4\}$ of $S$ with $r_1<r_2<r_3<r_4$. If $\{r_1,r_2,r_3,r_4\}\cap [0,1)\neq \emptyset$, then $(\{r_1,\dots,r_4\},r_3)$ cannot be AM. Thus we may assume that $\{r_1,r_2,r_3,r_4\}\cap [0,1)=\emptyset$.

Note that by the choice of $K$ we have
\[Kr_2>
r_2+\frac{p_4-p_2}{p_2-p_1} r_2=
r_2+\frac{p_4-p_2}{p_2-p_1} \left (\lambda(p_2-p_1)+r_1\right )\ge
r_2+\lambda (p_4-p_2)=
r_4.\]
Hence $\{r_2,r_3,r_4\}$ is contained in the union of two consecutive intervals of the form $[K^i,K^{i+1})$. This means that $(\{r_1,\dots,r_4\},r_3)$ cannot be AM since either $\{r_2,r_3,r_4\}$ is monochromatic, or $r_2$ and $r_4$ have different colours.

\medskip
    
\bf Construction of $\bf \varphi_2$ ($\bf s_0=p_2$):  \rm Fix $K$ such that $K>\frac{p_4-p_2}{p_2-p_1}+1$, let $L=K\cdot \left \lceil \frac{p_3-p_1}{p_4-p_3}\right \rceil$ and let $\{0,\dots,2L\}$ be the set of colours. We define $\varphi_2$ as follows. For each odd $i\in \N\cup \{0\}$, divide the interval $[L\cdot K^i,L\cdot K^{i+1})$ into $L$ equal half-closed intervals, and colour the $j$-th of them with colour $j$. For even $i\in \N\cup \{0\}$ divide the interval $[L\cdot K^i,L\cdot K^{i+1})$ into $L$ equal half-closed intervals, and colour the $j$-th of them with colour $L+j$. That is, for $j=1,\dots,L$ we colour $[L\cdot K^i+(j-1)(K^{i+1}-K^i),L\cdot K^i+j(K^{i+1}-K^i))$ with colour $j$ if $i$ is odd, and with colour $j+L$ if $i$ is even. Finally, colour the points in $[0,L)$ with colour $0$. 
    
$\varphi_2$ defined this way is AP-free, since it contains arbitrarily long monochromatic intervals of colours $1,\dots,2L$. Thus we only have to show it does not contain an AM positive homothetic copy of $(S,s_0)$. 

Consider a positive homothet $c+\lambda S=\{r_1,r_2,r_3,r_4\}$ of $S$ with $r_1<r_2<r_3<r_4$. If $\{r_1,r_2,r_3,r_4\}\cap [0,L)\neq \emptyset$, then $(\{r_1,r_2,r_3,r_4\},r_2)$ cannot be AM, thus we may assume that $\{r_1,r_2,r_3,r_4\}\cap [0,L)=\emptyset$. Note that by the choice of $K$ we again have 
\[Kr_2>
r_2+\frac{p_4-p_2}{p_2-p_1} r_2=
r_2+\frac{p_4-p_2}{p_2-p_1} (\lambda(p_2-p_1)+r_1)\ge
r_2+\lambda (p_4-p_2)=
r_4.\]
This means that $\{r_2,r_3,r_4\}$ is contained in the union of two consecutive intervals of the form $[L\cdot K^i,L\cdot K^{i+1})$, which implies that if $\left (\{r_1,r_2,r_3,r_4\},r_2\right )$ is AM, then $\{r_3,r_4\}$ is contained in an interval $[L\cdot K^i+(j-1)(K^{i+1}-K^i),L\cdot K^{i}+j(K^{i+1}-K^i))$ for some $1\leq j \leq L$. However, then by the choice of $L$ we have that $r_1$ is either contained in the interval $[L\cdot K^i,L\cdot K^{i+1})$ or in the interval $[L\cdot K^{i-1},L\cdot K^i)$. Indeed, 
\begin{align*}
    r_3-r_1&\leq \left \lceil \frac{r_3-r_1}{r_4-r_3}\right \rceil (r_4-r_3) \leq \left \lceil \frac{r_3-r_1}{r_4-r_3}\right \rceil (K^{i+1}-K^i)\\
    &=  \left \lceil \frac{p_3-p_1}{p_4-p_3}\right \rceil (K^{i+1}-K^i)
=L(K^i-K^{i-1}).
\end{align*}

Thus, if $r_1$ has the same colour as $r_3$ and $r_4$, then $r_1$ is also contained in the \mbox{interval} \mbox{$[L\cdot K^i+(j-1)(K^{i+1}-K^i),L\cdot K^{i}+j(K^{i+1}-K^i))$}, implying that $\left(\{r_1,r_2,r_3,r_4\},r_2\right)$ is mo\-no\-chro\-ma\-tic.

\medskip
We now extend the colouring to $\mathbb{R}$ in the case of $s_0=p_3$. Let $\varphi_2'$ be a colouring of $\mathbb{R}_{\leq 0}$ isometric to the reflection of $\varphi_2$ over $0$. Then $\varphi_2'$ contains no AM positive homothet of $(S,s_0)$.
If further we assume that $\varphi_1$ and $\varphi_2'$ use disjoint sets of colours, then the union of $\varphi_1$ and $\varphi_2'$ is an AP-free colouring of $\mathbb{R}$ containing no AM positive homothet of $(S,s_0)$.\\
We can extend the colouring similarly in the case of $s_0=p_2$.
\end{proof}

\section{Higher dimensions}\label{3}
In this section we prove Theorem \ref{introZ}.

\subsection{Proof of `if' direction of Theorem~\ref{introZ}}\label{32}
Let $S\subseteq \mathbb{R}^d$ such that $|S|>3$ and $s_0$ is not an extreme point of $S$. To prove the ‘if’ direction of Theorem~\ref{introZ}, we prove that there is an AP-free colouring of $\mathbb{R}^d$ without an AM similar copy of $(S,s_0)$. (Note that for the proof of Theorem~\ref{introZ}, it would be sufficient to prove this for $S\subseteq \mathbb{Z}^d$.)

Recall that $C\subseteq \mathbb{R}^d$ is a \emph{convex cone} if for every $x,y\in C$ and $\alpha,\beta\geq 0$, the vector $\alpha x+\beta y$ is also in $C$.
The \emph{angle} of $C$ is $\sup_{x,y\in C\setminus\{o\}}\angle(x,y)$.

We partition $\R^d$ into finitely many convex cones $C_1\cup\dots\cup C_m$, each of angle at most $\alpha=\alpha(d,S)$, where $\alpha(d,S)$ will be set later. We colour the cones with pairwise disjoint sets of colours as follows. First, we describe a colouring $\varphi$ of the closed circular cone $C=C(\alpha)$ of angle $\alpha$ around the line $x_1=\dots=x_d$. Then for each $i$ we define a colouring $\varphi_i$ of $C_i$ using pairwise disjoint sets of colours in a similar way. More precisely, let $f_i$ be an isometry with $f_i(C_i)\subseteq C$, and define $\varphi_i$ such that it is isometric to $\varphi$ on $f_i(C_i)$.

It is not hard to see that it is sufficient to find an AP-free colouring $\varphi$ of $C$ without an AM similar copy of $(S,s_0)$. Indeed, since the cones $C_i$ are coloured with pairwise disjoint sets of colours, any AP or AM similar copy of $(S,s_0)$ is contained in one single $C_i$.

We now turn to describing the colouring $\varphi$ of $C$.
Note that by choosing $\alpha$ sufficiently small we may assume that $C\subseteq \mathbb{R}_{\geq 0}^d$.
For $x\in \mathbb R^d$ let $\|x\|_1=|x_1|+\dots+|x_d|$.  Then for any $x\in \mathbb{R}^d$ we have
\begin{equation}\label{upperf}
\|x\|\leq \|x\|_1\leq \sqrt{d}\|x\|.
\end{equation}

Let $S=\{p_1,\dots,p_n\}$ and fix $K$ such that  \[K>1+2\sqrt{d}\max_{p_i,p_j,p_l,p_{\ell}\in S, p_i\neq p_j}\frac{\|p_k-p_\ell\|}{\| p_i-p_j \|}.\] 
For a sufficiently large $L$, to be specified later, we define $\varphi: C\to \{0,1,\dots,2L\}$ as

\begin{equation*}\def\stackalignment{l}
    \varphi(x) =
    \begin{cases*}
      0 & if $\|x\|_1<L$ \\
      j & \stackunder{if for some even $i\in \N$ and $j\in [L]$  we have}{$\|x\|_1\in [L\cdot K^i+(j-1)(K^{i+1}-K^i),L\cdot K^i+j(K^{i+1}-K^i))$}\\
     L+j & \stackunder{if for some odd $i\in \N$ and $j\in [L]$ we have}{$\|x\|_1\in [L\cdot K^i+(j-1)(K^{i+1}-K^i),L\cdot K^i+j(K^{i+1}-K^i))$.}
    \end{cases*}
\end{equation*}

$\varphi$ is AP-free since any halfline in $C$ contains arbitrarily long monochromatic sections of colours $1,\dots,2L$. Thus we only have to show that it does not contain an AM similar copy of $(S,s_0)$. 
Let $(\{r_1,\dots,r_n\},q_0)$ be a similar copy of $(S,s_0)$, with 
\begin{equation}\label{orderf}\|r_1\|_1\leq \|r_2\|_1\leq \dots \leq \|r_n\|_1.
\end{equation}

\begin{claim}\label{consecutive} $\{\|r_2\|_1,\dots,\|r_n\|_1\}$ is contained in the union of two consecutive intervals of the form $[L\cdot K^j, L\cdot K^{j+1})$.
\end{claim}
\begin{proof}
For any $r_i$ with $i\geq 2$ we have

\begin{flalign*}
\hspace{2cm}\|r_i\|_1 & \le  \|r_2\|_1+\|r_i-r_2\|_1  & \\
\hspace{2cm} & \leq  \|r_2\|_1+\sqrt{d}\|r_i-r_2\| & \text{by \eqref{upperf}} \\
\hspace{2cm}& =  \|r_2\|_1+\sqrt{d}\frac{\|r_i-r_2\|}{\|r_2-r_1\|}\|r_2-r_1\|
  &  \\ 
 \hspace{2cm} & < \|r_2\|_1+\frac{K-1}{2} \cdot \|r_2-r_1\|  & \text{by the definition of $K$} \\
\hspace{2cm}& \leq \|r_2\|_1+\frac{K-1}{2}(\|r_2\|+\|r_1\|) & \text{by the triangle inequality}\\
\hspace{2cm}& \leq \|r_2\|_1+\frac{K-1}{2}(\|r_2\|_1+\|r_1\|_1) & \text{by \eqref{upperf}} \\
\hspace{2cm}& \leq K\|r_2\|_1 & \text{by \eqref{orderf}}.
\end{flalign*}
\end{proof}

Assume now that $(\{r_1,\dots,r_n\},q_0)$ is AM. Note that $\|x\|_1$ is $\sqrt{d}$ times the length of the projection of $x$ on the $x_1=\dots=x_d$ line for $x\in \mathbb{R}_{\geq 0}^d$. Thus for any similar copy $\psi(S)$ of $S$ we have $\|\psi(s_0)\|_1\in \conv \setbuilder{\|p\|_1}{p\in \psi(S\setminus\{s_0\})}$, and we know that $q_0\neq r_1,r_n$ because $s_0$ is not an extreme point of $S$. This means that $\varphi(r_1)=\varphi(r_n)$, and there is exactly one $h\in \{2,\dots,n-1\}$ with $\varphi(r_h)\neq \varphi(r_1)$. 
This, by Claim \ref{consecutive} and by the definition of $\varphi$, is only possible if $h=2$ and there are $i\in \mathbb{N}$ and $j\in [L]$ such that 
\[\|r_3\|_1,\dots,\|r_{n-1}\|_1\in [L\cdot K^i+(j-1)(K^{i+1}-K^i),L\cdot K^i+j(K^{i+1}-K^i)).\]

The following claim finishes the proof.
\begin{claim}\label{r1}
If $L$ is sufficiently large and $\alpha$ is sufficiently small, then $\|r_1\|_1$ is contained in $[L\cdot K^{i-1}, L\cdot K^i)\cup [L\cdot K^{i}, L\cdot K^{i+1})$.
\end{claim}

The claim indeed finishes the proof. By the definition of $\varphi$, then $\varphi(r_1)=\varphi(r_n)$ implies \[\|r_1\|_1 \in [L\cdot K^i+(j-1)(K^{i+1}-K^i),L\cdot K^i+j(K^{i+1}-K^i)).\] But then we have \[\|r_2\|_1\in [L\cdot K^i+(j-1)(K^{i+1}-K^i),L\cdot K^i+j(K^{i+1}-K^i))\] as well, contradicting $\varphi(r_2)\neq \varphi(r_1)$.

\medskip

\begin{proof}[Proof of Claim~\ref{r1}] It is sufficient to show that $\|r_{n-1}\|_1-\|r_1\|_1<LK^i-LK^{i-1}$. We have
\[
    \|r_{n-1}\|_1-\|r_1\|_1\leq \sqrt{d}\|r_{n-1}-r_1\|=\sqrt{d}\|r_{n}-r_{n-1}\|\frac{\|r_{n-1}-r_n\|}{\|r_n-r_{n-1}\|}<  \|r_{n-1}-r_n\|\frac{K-1}{2},
\]
by \eqref{upperf} and by the definition of $K$. Let $H_1$ and $H_2$ be the hyperplanes orthogonal to the line $x_1=\dots=x_d$ at distance $\frac{1}{\sqrt{d}}(L\cdot K^i+(j-1)(K^{i+1}-K^i))$ and $\frac{1}{\sqrt{d}}(L\cdot K^i+j(K^{i+1}-K^i))$ from the origin respectively. Since $\|r_n\|_1, \|r_{n-1}\|_1\in [L\cdot K^i+(j-1)(K^{i+1}-K^i),L\cdot K^i+j(K^{i+1}-K^i))$ we have that $r_n$ and $r_{n-1}$ are contained in the intersection $T$ of $C(\alpha)$ and the slab bounded by the hyperplanes $H_1$ and $H_2$. 

Thus $\|r_{n-1}-r_n\|$ is bounded by the length of the diagonal of the trapezoid which is obtained as the intersection of $T$ and the $2$-plane through $r_n$, $r_{n-1}$ and the origin. Scaled by $\sqrt{d}$, this is shown in Figure~\ref{triangle}.

\begin{figure}[h]
\centering
\hspace{-1cm}
{\includegraphics[width=14.5cm]{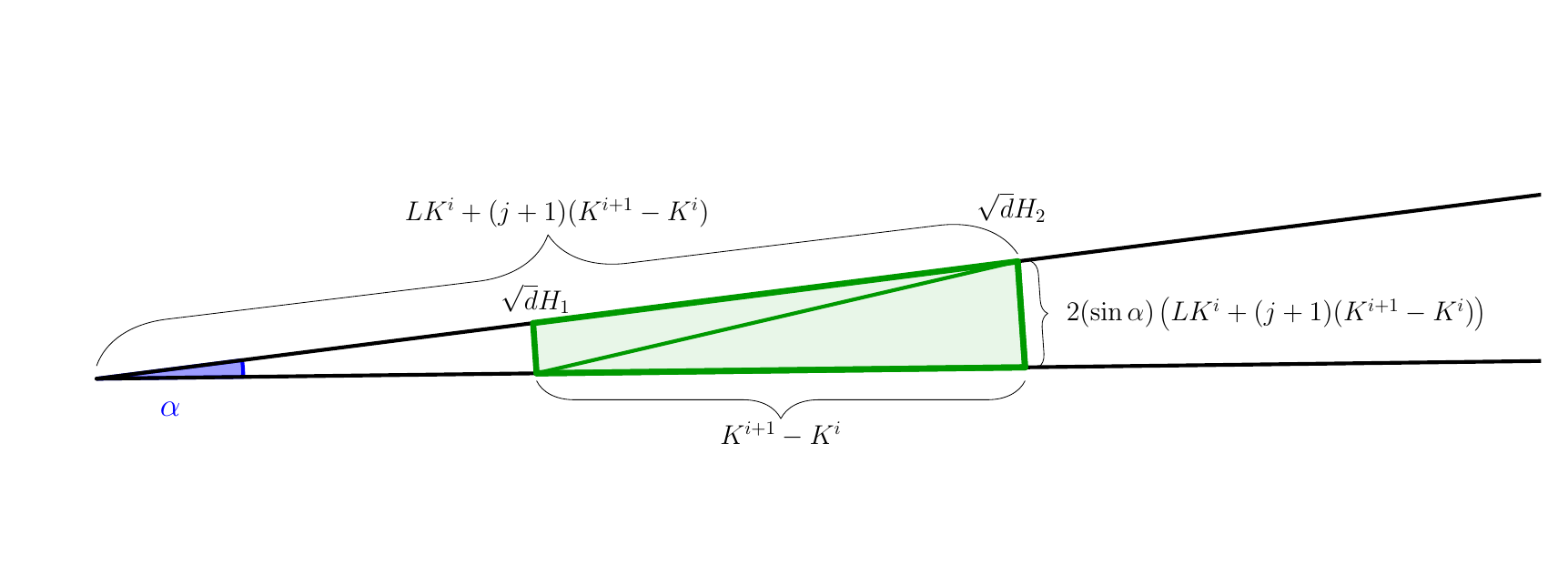}}
\caption{$T\cap C(\alpha)$}\label{triangle}
\end{figure}

From this, by the triangle inequality we obtain
\begin{multline*}
\|r_{n-1}-r_n\|\leq  \frac{1}{\sqrt{d}}\left (K^{i+1}-K^i+2(\sin\alpha)\left (LK^i+(j+1)(K^{i+1}-K^i)\right )\right )\\ \leq \frac{1}{\sqrt{d}} \left (K^{i+1}-K^i+2\sin\alpha \cdot LK^{i+1}\right )\leq \frac{2}{\sqrt{d}}\left (K^{i+1}-K^i\right ),
\end{multline*}
where the last inequality holds if $\alpha$ is sufficiently small. Combining these inequalities and choosing $L=\frac{K^2}{\sqrt{d}}$ we obtain the desired bound $\|r_{n-1}\|_1-\|r_n\|_1<LK^i-LK^{i-1}$, finishing the proof of the claim.
\end{proof}

\subsection{Proof of `only if' direction of Theorem \ref{introZ}}

The `only if' direction follows from Theorem \ref{introsimilar1} in the case of $d=1$, and from the following stronger statement for $d\geq 2$ (since in this case $s_0$ is an extreme point of $S$).

\begin{theorem}\label{grid} Let $S\subseteq \mathbb{Z}^d$ and $s_0\in S$ be an extreme point of $S$. Then for every $k$ there is a constant $\Lambda=\Lambda(d,S,k)$ such that the following is true. Every $k$-colouring of $\mathbb{Z}^d$ contains either an AM similar copy of $(S,s_0)$ or a monochromatic similar copy of $\Z^d$ with an integer scaling ratio $1\leq \lambda \leq \Lambda$.
\end{theorem}

Before the proof we need some preparation.

\begin{lemma}\label{lemextension}
	 There is an $R>0$ such that for any ball $D$ of radius at least $R$ the following is true. For every $p\in \Z^d$ outside $D$ and is at distance at most $1$ from $D$ there is a similar copy $(S',s_0')$ of $(S,s_0)$ in $\mathbb{Z}^d$ such that $s_0'=p$ and $S'\setminus \{s_0'\}\subset D$.
\end{lemma}




\begin{proof}Since $s_0$ is an extreme point of $S$, there is a hyperplane that separates $s_0$ from $S\setminus\{s_0\}$. Thus if $R$ is sufficiently large, there is an $\varepsilon>0$ with the following property. If $p$ is outside $D$ and is at distance at most $1$ from $D$, then there is a congruent copy $(S'',s_0'')$ of $(S,s_0)$ with $s_0''=p$ and such that every point of $S''\setminus \{s_0''\}$ is contained in $D$ at distance at least $2\varepsilon$ from the boundary of $D$. 

Let $\mathbb{Q}_N=\setbuilder{\frac{a}{b}}{a,b\in \mathbb{Z}, b\leq N}\subseteq \mathbb{Q}$. We use the fact that $O(\mathbb{R}^d)\cap \mathbb{Q}^{d\times d}$, the set of rational rotations, is dense in $O(\mathbb{R}^d)$ (see for example \cite{dense}). This, together with the compactness of balls implies that we can find an $N=N(\varepsilon)\in \mathbb{N}$ and $(S^*,s_0'*)$ in $\mathbb{Q}_N^d$ which is a rotation of $(S'',s_0'')$ around $p$, $\varepsilon$-close to $(S'',s_0'')$.  With this $S^*\setminus \{s_0^*\}$ is contained in $D$. Moreover, if $R$ is sufficiently large, then dilating $(S^*,s_0'*)$ from $s_0'*$ by $N!$, $S'\setminus \{s_0'\}$ is contained in $D\cap \mathbb{Z}^d$.
\end{proof}

The proof of the following variant of Gallai's theorem can be found in the Appendix.

\begin{theorem}[Gallai]\label{gallaiZ} Let $S\subseteq\mathbb{Z}^d$  be finite. Then there is a $\lambda(d,S,k)\in \mathbb{Z}$ such that every $k$-colouring of $\mathbb{Z}^d$ contains a monochromatic positive homothet of $S$ with an integer scaling ratio bounded by $\lambda(d,S,k)$.
\end{theorem}

\begin{proof}[Proof of Theorem \ref{grid}]

Let $R$ be as in Lemma~\ref{lemextension} and let $H$ be the set of points of $\Z^d$ contained in a ball of radius $R$. By Theorem~\ref{gallaiZ} there is a monochromatic, say blue, homothetic copy $H_0=c+\lambda H$ of $H$ for some integer $\lambda\leq\lambda(d,H,k)$. Without loss of generality we may assume that $H_0=B(O, \lambda R)\cap \lambda\mathbb{Z}^d$ for some $O\in \mathbb{Z}^d$, where $B(O, \lambda R)$ is the ball of radius $\lambda R$ centred at $O$.

Consider a point $p\in \lambda \Z^d\setminus H$ being at distance at most $\lambda$ from $H_0$. If $p$ is not blue then using Lemma~\ref{lemextension} we can find an AM similar copy of $(S,s_0)$. Thus we may assume that any point $p\in \lambda \Z^d\setminus H_0$ which is $\lambda$ close to $H_0$ is blue as well. 

By repeating a similar procedure, we obtain that there is either an AM similar copy of $(S,s_0)$, or every point of $H_i=B(O,\lambda R+i \lambda )\cap \lambda\Z^d$ is blue for every $i\in \N$. But the latter means $\lambda \Z^d$ is monochromatic, which finishes the proof.
\end{proof}

\subsection{Finding an AM positive homothet}

The following statement shows that it is not possible to replace an AM similar copy of $(S,s_0)$ with a positive homothet of $(S,s_0)$ in the `only if' direction of Theorem~\ref{introZ}.

\begin{proposition}\label{introsimilar2} Let $S\subseteq\mathbb{Z}^d$ such that $S$ is not contained in a line and $s_0\in S$.
Then there is an AP-free colouring of $\mathbb{R}^d$ without an AM positive homothet of $(S,s_0)$.
\end{proposition}
\begin{figure}
\centering
\hspace{-1cm}
{\includegraphics[width=14.5cm]{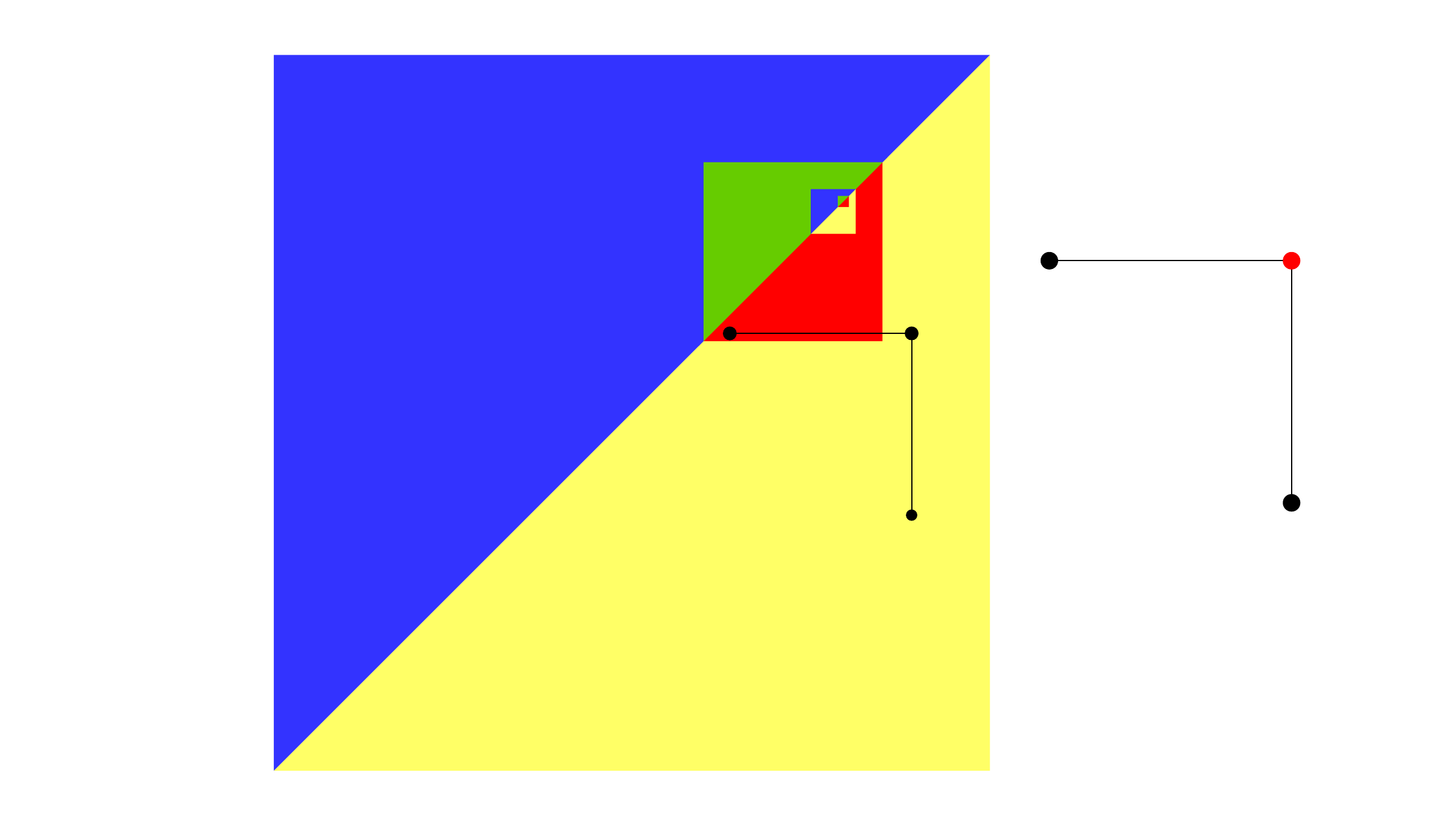}}
\caption{A $4$-colouring avoiding AM homothets of $(S,s_0)$.}\label{mondrian}
\end{figure}
\begin{proof}
We may assume that $|S|=3$ and thus $S\subseteq \mathbb{R}^2$.
Since the problem is affine invariant, we may  further assume that $S=\{(0,1),(1,1),(1,0)\}$ with $s_0=(1,1)$, $s_1=(0,1)$ and $s_2=(1,0)$. First we describe a colouring of $\mathbb{R}^2$ and then we extend it to $\mathbb{R}^d$.

For every $i\in\N$ let $Q_i$ be the square $[-4^i,4^{i-1}]\times [-4^i,4^{i-1}]$, and $Q_0=\emptyset$. Further let $H_+$ be the open halfplane $x>y$ and $H_-$ be the closed halfplane $x\leq y$.  We colour $\mathbb{R}^2$ using four colours, green, blue, red and yellow as follows (see also Figure \ref{mondrian}).
\begin{itemize}
    \item \bf Red: \rm For every odd $i\in \N$ colour $(Q_i\setminus Q_{i-1})\cap H_+$ with red.
    \item \bf Yellow: \rm For every even $i\in \N$ colour $(Q_i\setminus Q_{i-1})\cap H_+$ with yellow.
    \item \bf Green: \rm For every odd $i\in \N$ colour $(Q_i\setminus Q_{i-1})\cap H_-$ with green.
    \item \bf Blue: \rm For every even $i\in \N$ colour $(Q_i\setminus Q_{i-1})\cap H_-$ with blue.
\end{itemize}

A similar argument that we used in the proof of Lemma \ref{1notextreme} shows that this colouring $\varphi_1$ is AP-free. Thus we only have to check that it contains no AM positive homothet of $(S,s_0)$. Let $S'$ be a positive homothet of $S$. First note that we may assume that $S'$ is contained in one of the halfplanes bounded by the $x=y$ line, otherwise it is easy to see that it cannot be AM. Thus by symmetry we may assume that $s_0'\in Q_i\setminus Q_{i-1}\cap H_+$ for some $i\in \N$.

If the $x$-coordinate of $s_0'$ is smaller than $4^{i-2}$ or the $y$-coordinate of $s_0'$ is smaller than $-4^{i-1}$, then $s_1'\in H_+$ implies $s_1'\in Q_i\setminus Q_{i-1}$, and hence $S'$ cannot be AM.
Otherwise, necessarily $\|s_0'-s_1'\|\leq 2\cdot 4^{i-1}$.
Since $\|s_0'-s_1'\|=\|s_0'-s_2'\|$, this means that the $y$-coordinate of $s_2'$ is at least $-4^{i-1}-2\cdot 4^{i-1}>-4^i$. Thus, in this case $s_2'$ is contained in $Q_i\setminus Q_{i-1}$, and hence $S'$ cannot be monochromatic.

To finish the proof, we extend the colouring to $\mathbb{R}^d$. Let $T \cong \mathbb{R}^{d-2}$ be the orthogonal complement of $\R^2$. Fix an AP-free colouring $\varphi$ of $T$ using the colour set $\{1,2\}$. Further let $\varphi_2$ be a colouring of $\mathbb{R}^2$ isometric to $\varphi_1$, but using a disjoint set of colours. For every $t\in T$ colour $\mathbb{R}^2+t$ by translating $\varphi_i$ if $\varphi(t)=i$. This colouring is AP-free and does not contain any AM positive homothet of $(S,s_0)$.
\end{proof}

\section{Colouring $\mathbb{Q}$}\label{inQ}

Before proving Theorem \ref{iQ} we note that we could not replace $S=\{0,1,2\},s_0=0$ with any arbitrary pair $(S,s_0)$ where $s_0$ is an extreme point of $S$. For example let $S=\{0,1,2,3,4\}$, $s_0=0$, and colour $\mathbb{Q}$ as follows.
Write each non-zero rational as $2^t\frac pq$ where $p$ and $q$ are odd, and colour it red if $t$ is even and blue if $t$ is odd. This colouring is non-monochromatic, but does not contain any AM homothet of $(S,s_0)$ (in fact not even similar copies).
	
Indeed, let $c+\lambda S=\{r_1,r_2,r_3,r_4,r_5\}=\{r_1,r_1+\lambda,r_1+2\lambda, r_1+3\lambda, r_1+4\lambda\}$ be a homothet of $S$. Note that for any $x,y,\alpha\in \mathbb{Q}$ we have that $x$ and $y$ are of the same colour if and only if $\alpha x$ and $\alpha y$ have the same colour. That is, multiplying each $r_i$ with the same $\alpha$ does not change the colour pattern. Thus we may assume that $r_1=2^ab$ and $\lambda=2^tp$ for some $a,t\in \mathbb{N}$ and odd integers $b,p$.
	
If $a<t$, then $\{r_1,r_2,r_3,r_4,r_5\}$ is monochromatic, thus we may assume that $t\leq a$ and divide by $2^t$, to obtain $\{2^{a-t}b,2^{a-t}b+p,2^{a-t}b+2p,2^{a-t}b+3p,2^{a-t}b+4p\}\subseteq \mathbb{Z}$. But then two of $2^{a-t}b+p,2^{a-t}b+2p,2^{a-t}b+3p,2^{a-t}b+4p$ are odd, and one of them is $2 \bmod 4$, implying that $\{2^{a-t}b+p,2^{a-t}b+2p,2^{a-t}b+3p,2^{a-t}b+4p\}$ cannot be monochromatic.

    \bigskip

To prove Theorem~\ref{iQ} we need the following lemma.

\begin{lemma}\label{lemZ}
If a $k$-colouring $\varphi$ of $\Z$ does not contain an AM positive homothet of $(S,s_0)$, where $S=\{0,1,2\}$ and $s_0=0$, then every colour class is a two-way infinite AP. Moreover, there is an $F=F(k)$ such that $\varphi$ is periodic with $F$.
\end{lemma}

\begin{proof}[Proof of Theorem \ref{iQ}] Let $\varphi_\Q$ be a colouring of $\mathbb{Q}$ without an AM positive homothet of $(S,s_0)$. For any $x\in \mathbb{Q}$, we define a colouring $\varphi$ of $\Z$ as $\varphi(n)=\varphi_\Q(\frac{xn}{F(k)})$. Applying Lemma~\ref{lemZ} for $\varphi$ implies that $\varphi_\Q(0)=\varphi(0)=\varphi(F(k))=\varphi_\Q(x)$. this means $0$ and $x$ have
the same colour in $\varphi_\Q$. 
\end{proof}

\begin{proof}[Proof of Lemma~\ref{lemZ}] 
We first prove that if $a$ and $a+d$ for some $d>0$ have the same colour, say red, then every point in the two-way infinite AP $\{a+id\mid i\in \Z\}$ is also red. 
    
For this first we show that $\{a-id\mid i\in \N\}$ is red. Indeed, if it is not true, let $j\in \mathbb{N}$ be the smallest such that $a-jd$ is not red. But then $\{a-jd,a-(j-1)d,a-(j-2)d\}$ is AM, a contradiction. 
    
Second, we show that $\{a+id\mid i\in \N\}$ contains at most one element of any other colour. Assume to the contrary that there are at least two blue elements in $\{a+id\mid i\in \N\}$, and let $a+jd$, $a+kd$ be the two smallest with $j<k$. But then since $a+jd-(k-j)d$ is red, we have that $\{a+jd-(k-j)d, a+jd, a+kd\}$ is AM, a contradiction. 
    
Finally we show that every element $\{a+id\mid i\in \N\}$ is red. Assume that $(a+jd)$ is the largest non-red element. (By the previous paragraph this largest $j$ exists.) But then, since $a+(j+1)d$ and $a+(j+2)d$ are red, we obtain that $\{a+jd,a+(j+1)d, a+(j+2)d\}$ is AM, a contradiction.

If $d$ is the smallest difference between any two red numbers, this shows that the set of red numbers is a two-way infinite AP. Thus to finish the first half of the claim, we only have to show that if a colour is used once, then it is used at least twice. But this follows from the fact that the complement of the union of finitely many AP is either empty or contains an infinite AP.

In order to prove the second part, it is sufficient to show that there is an $N_k$ depending on $k$ such that the following is true. If $\mathbb{Z}$ is covered by $k$ disjoint AP, then the difference of any of these AP's is at most $N_k$. Thus, by considering densities, the following claim finishes the proof of Lemma~\ref{lemZ}.
    
\begin{claim} There is an $N_k$ such that if for $x_1,\dots,x_k\in \mathbb{N}$ we have $\sum_{i=1}^k\frac{1}{x_i}=1$, then $x_i\leq N_k$ for all $1\leq i\leq k$.
\end{claim}
    
\begin{proof} We prove by induction on $k$ that for every $c\in \mathbb R^+$ there is a number $N_k(c)$ such that if we have $\sum_{i=1}^k\frac{1}{x_i}=c$, then $x_i\leq N_k(c)$ for all $1\leq i\leq k$.
For $k=1$ setting $N_1(c)=\lfloor \frac 1c\rfloor$ is a good choice.

For $k>1$ notice that the \emph{smallest} number whose reciprocal is in the sum is at most $\lfloor \frac kc\rfloor$. Thus we obtain $N_k(c)\le \max_{i=1}^{\lfloor \frac kc\rfloor}\left(\max(i;N_{k-1}(c-\frac 1i)\right)$.
\end{proof}
This finishes the proof of Lemma~\ref{lemZ}, and thus also of Theorem \ref{iQ}.
\end{proof}

It would be interesting to find a complete characterisation of those pairs $(S,s_0)$ for which Theorem \ref{iQ} holds.

\begin{question}
For which $(S,s_0)$ is it true that every finite colouring of $\mathbb{Q}$ with more than one colour contains an AM positive homothet of $(S,s_0)$?
\end{question}

\section{Bold bouquets and the chromatic number of the plane}\label{4}


Suppose that for a graph $G=(V,E)$ with a given origin (distinguished vertex) $v_0\in V$, we have a colouring $\varphi$ where $\varphi:V\setminus\{v_0\} \to \{1,\dots,k\}$ and $\varphi(v_0)\in \binom{\{1,\dots,k\}}{2}$, i.e., each vertex apart from $v_0$ gets one color from the set $\{1,\dots,k\}$, while $v_0$ get two colors from \{1,\dots,k\}.
We say that $\varphi$ is a proper $k$-colouring with \emph{bichromatic origin} $v_0$, if $(v,w)\in E$ implies $\varphi(v)\cap \varphi(w)=\emptyset$. There are unit-distance graphs with not too many vertices that do not have a $4$-colouring with a certain bichromatic origin. Figure~\ref{abra4} shows such an example, the $34$-vertex graph $G_{34}$, posted by Hubai \cite{tamas} in Polymath16. Finding such graphs has been motivated by an approach to find a human-verifiable proof of $\chi(\mathbb{R}^5)\geq 5$, proposed by the Pálvölgyi \cite{domotor} in Polymath16.

$G_{34}$ is the first example  found whose chromatic number can be verified quickly without relying on a computer. To see this, note that the vertices connected to the central vertex have to be coloured with two colours, and they can be decomposed into three $6$-cycles. Using this observation and the symmetries of the graph, we obtain that there are only two essentially different ways to colour the neighbourhood of the central vertex. In both cases, for the rest of the vertices a systematic back-tracking strategy shows in a few steps that there is no proper colouring with four colours.

\begin{figure}[h]
\begin{center}
\input{abra4.inc}
\end{center}
\caption{\small{A $34$ vertex graph without a $4$-colouring if the origin is bichromatic.}}\label{abra4}
\end{figure}
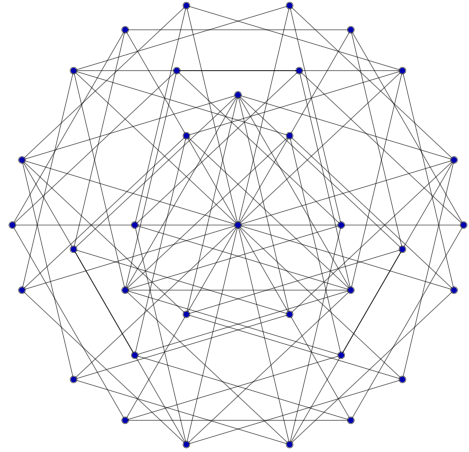

Theorem~\ref{human} with $G=G_{34}$ shows that a human-verifiable proof of Conjecture~\ref{introbouquet} for $k=4$ would provide a human-verifiable proof of $\chi(\R^2)\geq 5$. Note that $G_{34}$ was found by a computer search, and for finding other similar graphs one might rely on a computer program. Thus, the approach we propose, is human-verifiable, however it might be computer-assisted.

For a graph $G$ with origin $v_0$ let $\{C_1,\dots,C_n\}$ be the set of unit circles whose centres are the neighbours of $v_0$, and let $C(G,v_0)=C_1\cup\dots\cup C_n$ be the bouquet through $v_0$.

\begin{theorem}\label{human} If there is a unit-distance graph $G=(V,E)$ with $v_0\in V$ which does not have a proper $k$-colouring with bichromatic origin $v_0$, and Conjecture~\ref{introbouquet} is true for $C(G,v_0)$, then $\chi(\R^2)\geq k+1$.
\end{theorem}

\begin{proof}[Proof of Theorem~\ref{human}] Assume for a contradiction that there is a proper $k$-colouring $\varphi$ of the plane. Using $\varphi$ we construct a proper $k$-colouring of $G$ with bichromatic origin $v_0\in V$.

Let $v_1,\dots,v_n$ be the neighbours of the origin $v_0$, and $C_j$ be the unit circle centred at $v_j$. Then $C=C_1\cup\dots\cup C_n$ is a bouquet through $v_0$.
If Conjecture~\ref{introbouquet} is true for $\varphi$, then there is a bold congruent copy $C'=C_1'\cup\dots C_n'$ of $C$ through $v_0'$. That is, there are points $p_1\in C_1',\dots,p_n\in C_n'$ with $\ell=\varphi(p_1)=\dots=\varphi(p_n)\neq \varphi(v_0')$.

For $i\in [n]$ let $v'_i$ be the centre of $C'_i$. We define a colouring $\varphi'$ of $G$ as $\varphi'(v_0)=\{\varphi(v_0'),\ell\}$ and $\varphi'(v_i)=\varphi(v_i')$ for 
$v\in V\setminus\{v_0\}$.
We claim that $\varphi'$ is a proper $k$-colouring of $G$ with a bichromatic origin $v_0$, contradicting our assumption.

Indeed, if $v_i\ne v_0\ne v_j$ then for $(v_i,v_j)\in E$ we have $\varphi'(v_i)\ne \varphi'(v_j)$ because $\varphi(v_i')\ne \varphi(v_j')$.
For $(v_0,v_i)\in E$, we have $\varphi'(v_i)\ne \varphi(v_0)$ because $\varphi(v_i')\ne \varphi(v_0')$, and $\varphi'(v_i)\ne \ell$ because $\varphi(v_i')\ne \ell$ since $\|v_i'-p_i\|=1$.
This finishes the proof of Theorem \ref{human}.
\end{proof}

\subsection{Bold pencils}\label{41}

In this section we prove Theorem \ref{intropencil}. The proof was originally an answer by the first author to a MathOverflow question of the third author \cite{FN}. We start with the following simple claim.

\begin{claim}\label{gyuru} For every pencil $L$ through $O$ there is an $\varepsilon> 0$ for which the following is true. For any circle $C$ of radius $R$ if a point $p$ is at distance at most $\varepsilon R$ from $C$, then there is a congruent copy $L'$ of $L$ through $p$ such that every line of $L'$ intersects $C$.
\end{claim}

\begin{proof} It is sufficient to prove the following. If $C$ is a unit circle and $p$ is sufficiently close to $C$, then there is a congruent copy $L'$ of $L$ through $p$ such that every line of $L'$ intersects $C$.

Note that if $p$ is contained in the disc bounded by $C$, clearly every line of every congruent copy $L'$ of $L$ through $p$ intersects $C$. Thus we may assume that $p$ is outside the disc. 

Let $0<\alpha<\pi$ be the largest angle spanned by lines in $L$. If $p$ is sufficiently close to $C$, then the angle spanned by the tangent lines of $C$ through $p$ is larger than $\alpha$. Thus, any congruent copy $L'$ of $L$ through $p$ can be rotated around $p$ so that every line of the pencil intersects $C$.
\end{proof}

\begin{proof}[Proof of Theorem~\ref{intropencil}]
Assume for contradiction that $\varphi$ is a colouring using at least two colours, but there is a pencil $L$ such that there is no congruent bold copy of $L$.

First we obtain a contradiction assuming that there is a monochromatic, say red, circle $C$ of radius $r$. We claim that then every point $p$ inside the disc bounded by $C$ is red. Indeed, translating $L$ to a copy $L'$ through $p$, each line $L_i'$ will intersect $C$, and so have a red point. Thus $p$ must be red.

A similar argument together with Claim~\ref{gyuru} shows that if there is a non-red point at distance at most $\varepsilon r$ from $C$, we would find a congruent bold copy of $L$ through $p$. Thus there is a circle $C'$ of radius $(1+\varepsilon)r$ concentric with $C$, such that every point of the disc bounded by $C'$ is red. Repeating this argument, we obtain that every point of $\mathbb{R}^2$ is red contradicting the assumption that $\varphi$ uses at least $2$ colours. 

To complete the proof, we show that there exists a monochromatic circle. For $1\leq i\leq n$ let $\alpha_i$ be the angle of $L_i$ and $L_{i+1}$. Fix a circle $C$, and let $a_1,\dots,a_n\in C$ be points such that if $c\in C\setminus \{a_1, \ldots, a_n\}$, then the angle of the lines connecting $c$ with $a_i$ and $c$ with $a_{i+1}$ is $\alpha_i$. By Gallai's theorem there is a monochromatic (say red) set $\{a'_1,\dots,a'_n\}$ similar to $\{a_1,\dots,a_n\}$. Let $C'$ be the circle that contains $\{a'_1,\dots,a'_n\}$. Then $C'$ is monochromatic. Indeed, if there is a point $p$ on $C'$ for which $\varphi(p)$ is not red, then by choosing $L_j'$ to be the line connecting $p$ with $a'_j$, we obtain $L'=L_1'\cup\dots\cup L_n'$, a bold congruent copy of $L$. 
\end{proof}

\subsection{Conjecture~\ref{introbouquet} for lattice-like bouquets}\label{42}

Using the ideas from the proof of Theorem \ref{grid}, we prove Conjecture \ref{introbouquet} for a broader family of bouquets.

\subsubsection{Lattices}

A \emph{lattice} $\L$ generated by linearly independent vectors $v_1,v_2\in \mathbb{R}^2$ is the set $\L=\L(v_1,v_2)=\setbuilder{n_1v_1+n_2v_2}{n_1,n_2\in \Z}$.
We call a lattice $\L$ \emph{rotatable} if for every $0\leq \alpha_1< \alpha_2\leq \pi$ there is an angle $\alpha_1<\alpha<\alpha_2$ and scaling factor $\lambda= \lambda(\alpha_2,\alpha_1)$ such that $\lambda\alpha(\L)\subset\L$, where $\alpha(L)$ is the rotated image of $\L$ by angle $\alpha$ around the origin.
For example, $\Z^2$, the triangular grid, and $\setbuilder{n_1(1,0)+n_2(0,\sqrt 2)}{n_1,n_2\in \Z}$ are rotatable, but $\L=\setbuilder{n_1(1,0)+n_2(0,\pi)}{n_1,n_2\in \Z}$ is not.\footnote{For another characterization of rotatable lattices, see \url{https://mathoverflow.net/a/319030/955}.} 

The rotatability of $\L$ allows us to extend Lemma \ref{lemextension} from $\mathbb{Z}^2$ to $\L$. This leads to an extension of Theorem~\ref{grid} to rotatable lattices.

\begin{theorem}\label{lattice}Let $\L$ be a rotatable lattice, $S\subseteq \L$ be finite and $s_0$ be an extreme point of $S$. Then for every $k\in N$ there exists a constant $\Lambda=\Lambda(\L,S,k)$ such that the following is true. In every $k$-colouring of $\L$ there is either an AM similar copy of $(S,s_0)$ with a positive scaling factor bounded by $\Lambda$, 
or a monochromatic positive homothetic copy of $\L$ with an integer scaling factor $1\leq \lambda \leq \Lambda$.
\end{theorem}

The proof of extending Lemma \ref{lemextension} to rotatable lattices is analogous to the original one, so is the proof of Theorem~\ref{lattice} to the proof of Theorem~\ref{grid}. Therefore, we omit the details.

\subsubsection{Lattice-like bouquets}

Let $C=C_1\cup\dots\cup C_n$ be a bouquet through $O$, and for $i\in [n]$ let $O_i$ be the centre of $C_i$. We call $C$ \emph{lattice-like} if $O$ is an extreme point of $\{O,O_1,\dots,O_n\}$ and there is a rotatable lattice $\L$ such that $\{O,O_1,\dots,O_n\}\subseteq \L$. Similarly, we call a unit-distance graph $G=(V,E)$ with an origin $v_0\in V$ \emph{lattice-like} if there is a rotatable lattice $\L$ such that $v_0$ and its neighbours are contained in $\L$, and $v_0$ is not in the convex hull of its neighbours.

Since $\mathbb{Z}^2$ is a rotatable lattice, Theorem \ref{introracbouq} is a direct corollary of the result below.

\begin{theorem}\label{latticeboq} If $C$ is a lattice-like bouquet, then every proper $k$-colouring of $\mathbb{R}^2$ contains a bold congruent copy of $C$.
\end{theorem}

This implies the following, similarly as Conjecture~\ref{introbouquet} implied Theorem~\ref{human}.

\begin{theorem}\label{implication} If there exists a lattice-like unit-distance graph $G=(V,E)$ with an origin $v_0$ that does not admit a proper $k$-colouring with bichromatic origin $v_0$, then $\chi(\R^2)\geq k+1$.
\end{theorem}

In the proof of Theorem \ref{latticeboq}, we need a simple geometric statement.

\begin{prop}\label{similar} Let $C=C_1\cup\dots\cup C_n$ be a bouquet through $O$, and let $\mathcal O=\{O_1,\dots,O_n\}$, where $O_j$ is the center of $C_j$. Then for every $0<\lambda\leq 2$ there are $n$ points $P_1,\dots,P_n$ such that $P_j\in C_j$ and $\{P_1,\dots,P_n\}$ is congruent to $\lambda \mathcal O$.
\end{prop}

\begin{proof} 
For $\lambda=2$ let $P_j$ be the image of $O$ reflected in $O_j$. Then $P_j\in C_j$, and $\{P_1,\dots,P_n\}$ can be obtained by dilating $\mathcal O$ from $O$ with a factor of $2$. For $\lambda<2$, scale $\{P_1,\dots,P_n\}$ by $\frac{\lambda}{2}$ from $O$ obtaining $\{P'_1,\dots,P'_n\}$. Then there is an angle $\alpha$ such that rotating $\{P'_1,\dots,P'_n\}$ around $O$ by $\alpha$, the rotated image of each $P'_j$ is on $C_j$. 
\end{proof}

\begin{proof}[Proof of Theorem~\ref{latticeboq}] Let $C=C_1\cup\dots\cup C_n$ be the lattice-like bouquet through $O$, $O_i$ be the centre of $C_i$ for $i\in[k]$, and $\L$ be the rotatable lattice containing $S=\{O,O_1,\dots,O_n\}$. Consider a proper $k$-colouring $\varphi$ of $\mathbb{R}^2$ and let $\delta\in \mathbb{Q}$ to be chosen later.

By Theorem~\ref{lattice}, the colouring $\varphi$ either contains an AM similar copy of $\delta (S,s_0)$  with a positive scaling factor bounded by $\lambda(\L,S,k)$, or a monochromatic similar copy of $\delta\L$ with an integer scaling factor bounded by $\lambda(\L,S,k)$.

If the first case holds and $\delta$ is chosen so that $\delta \lambda(\L,S,k)\leq 2$, Proposition~\ref{similar} provides a bold congruent copy of $C$. Now assume for contradiction that the first case does not hold. Then there is a monochromatic similar copy $\L'$ of $\delta \L$ with an integer scaling factor $\lambda$ bounded by $\lambda(\L,S,k)$. However, if we choose $\delta=\frac 1{\lambda(\L,S,k)!}$, then for any $1\leq \lambda\leq \lambda(\L,S,k)$ we have $\delta\lambda=\frac{1}{N_{\lambda}}$ for some $N_{\lambda}\in \mathbb{N}$. But this would imply that there are two points in the infinite lattice $\lambda\delta \L$ at distance $1$, contradicting that $\varphi$ is a proper colouring $\mathbb{R}^2$.
\end{proof}

\section{Further problems and concluding remarks}\label{5}

Problems in the main focus of this paper are about finding AM sets \emph{similar} to a given one. However, it is also interesting to find AM sets \emph{congruent} to a given one. In this direction, Erd\H os, Graham, Montgomery, Rothschild, Spencer and Straus made the following conjecture.

\begin{conjecture}[Erd\H os et al.~\cite{EGMRSS3}]\label{erdosconj}
Let $s_0\in S\subset\R^2$, $|S|=3$.
There is a non-monochromatic colouring of $\R^2$ that contains no AM congruent copy of $(S,s_0)$ if and only if $S$ is collinear and $s_0$ is not an extreme point of $S$.
\end{conjecture}

As noted in \cite{EGMRSS3}, the `if' part is easy; colour  $(x,y)\in\R^2$ red if $y>0$ and blue if $y\le 0$.
In fact, this colouring also avoids AM similar copies of such $S$.
Conjecture \ref{erdosconj} was proved in \cite{EGMRSS3} for the vertex set $S$ of a triangle with angles $120^\circ, 30^{\circ}$, and $30^{\circ}$ with any $s_0\in S$. It was also proved for any isosceles triangle in the case when $s_0$ is one of the vertices on the base, and for an infinite family of right-angled triangles.

Much later, the same question was asked independently in a more general form by the third author \cite{domotor2}.
In a comment to this question on the MathOverflow site, a counterexample (to both the MathOverflow question and Conjecture \ref{erdosconj}) was pointed out by user `fedja', which we sketch below.

\begin{counterexample}[`fedja' \cite{fedja}]
Let $S=\{0,1,s_0\}$ where $s_0\notin [0,1]$ is a transcendental number.
Then there is a field automorphism $\tau$ of $\mathbb C$ over $\Q$ such that $\tau(s_0)\in(0,1)$.
Pick an arbitrary halfplane $H\subset \mathbb C$, and colour $x$ red if $\tau(x)\in H$ and blue if $x\notin H$.
Suppose that there is an AM similar copy $\{a,a+b,a+bs_0\}$ of $(S,s_0)$.
Then these points are mapped by $\tau$ to $\tau(a)$, $\tau(a)+\tau(b)$ and $\tau(a)+\tau(b)\tau(s_0)$.
Since $\tau(s_0)\in(0,1)$, the point $\tau(a)+\tau(b)\tau(s_0)$ falls inside the planar segment $[\tau(a),\tau(a)+\tau(b)]$.
If $a$ and $a+bs_0$ have the same color, then either $\tau(a),\tau(a)+\tau(b)\in H$, or $\tau(a),\tau(a)+\tau(b)\notin H$.
But if $\tau(a),\tau(a)+\tau(b)\in H$, then $\tau(a)+\tau(b)\tau(s_0)\in H$, and similarly, if $\tau(a),\tau(a)+\tau(b)\notin H$, then $\tau(a)+\tau(b)\tau(s_0)\notin H$, which implies that $a+b$ gets the same color, 
so $\{a,a+b,a+bs_0\}$ cannot be an AM copy.\qed
\end{counterexample}

\bigskip

Straightforward generalisations of our arguments from Section \ref{4} would also imply lower bounds for the chromatic number of other spaces.
For example, if $C$ is a lattice-like bouquet of spheres, then every proper $k$-colouring of $\mathbb R^d$ contains a bold congruent copy of $C$. This implies that if one can find a lattice-like unit-distance graph with an origin $v_0$ that does not admit a proper $k$-colouring with bichromatic origin $v_0$, then $\chi(\mathbb R^d)\ge k+1$.
Possibly one can even strengthen this further; in $\mathbb R^d$ it could be even true that there is a \emph{$d$-bold} congruent copy of any bouquet $C$, meaning that there are $d$ colours that appear on each sphere of $C$.
This would imply $\chi(\mathbb R^d)\ge k+d-1$ if we could find a lattice-like unit-distance graph with an origin $v_0$ that does not admit a proper $k$-colouring with $d$-chromatic origin $v_0$.

\bigskip

In our AP-free colourings that avoid AM similar copies of certain sets, we often use many colours. 
It would be interesting to know if constructions with fewer colours exist, particularly regarding applications to the Hadwiger-Nelson problem.
In Lemma \ref{1notextreme} (and also in the colouring used for proving the `if' direction of Theorem \ref{introZ}) the number of colours we use is not even uniformly bounded. Are there examples with uniformly bounded number of colours?

\bigskip

One of our main questions is about characterising those pairs $(S,s_0)$ for which in every colouring of $\mathbb{R}^d$ we either find an AM similar copy of $(S,s_0)$ or an \emph{infinite} monochromatic AP. However, regarding applications to the Hadwiger-Nelson problem the following, weaker version would also be interesting to consider: Determine those $(S,s_0)$ with $S\subseteq \mathbb{R}^d$ and $s_0\in S$ for which there is a $D=D(k,S)$ such that the following is true. For every $n$ in every $k$-colouring of $\mathbb{R}^D$ there is an AM similar copy of $(S,s_0)$ or an $n$-term monochromatic AP with difference $t\in \mathbb{N}$ bounded by $D$. 
Note that there are pairs for which the property above does not hold when colouring $\mathbb{Z}$. For example let $S=\{-2,-1,0,1,2\}$, $s_0=0$, and colour $i\in \Z$ red if $\lfloor i/D\rfloor \equiv 0 \bmod 2$ and blue if $\lfloor i/D\rfloor \equiv 1 \bmod 2$.

\subsection*{Acknowledgements.} We thank Konrad Swanepoel, Peter Allen and an anonymous referee for helpful suggestions on improving the presentation of the paper, and the participants of the Polymath16 project for discussions and consent to publish these related results that were obtained `offline' and separately from the main project.

\bibliographystyle{amsplain}
\bibliography{biblio}

\appendix

\section{Proof of Gallai's theorem}\label{sec:gallai}

For completeness, we prove the stronger version of Gallai's theorem, Theorem \ref{gallaiZ}.
We use the Hales-Jewett theorem, following the proof from \cite{gallaiR}.

A \emph{combinatorial line} in $[n]^N$ is a a collection of $n$ points, $p_1, \ldots, p_n$, such that for some fix $I\subset [N]$ for every $i\in I$ the coordinate $(p_j)_i$ is the same for every $j$, while for $i\notin I$ the coordinate $(p_j)_i=j$ for every $j$.

\begin{theorem}[Hales-Jewett \cite{HalesJewett}]
    For every $n$ and $k$ there is an $N$ such that every $k$-colouring of $[n]^N$ contains a monochromatic combinatorial line.
\end{theorem}

\begin{proof}[Proof of Theorem \ref{gallaiZ}]
Suppose that we want to find a monochromatic positive homothet of some finite set $S=\{s_1,\ldots,s_n\}$ from $\Z^d$ in a $k$-colouring of $\Z^d$.
Choose an $N$ that satisfies the conditions of the Hales-Jewett theorem for $n$ and $k$.
Choose an injective embedding of $[n]^N$ into $\Z^d$ given by $\Psi(x_1,\ldots,x_N)=\sum_{i=1}^N \lambda_is_{x_i}$, where the $\lambda_i$'s are to be specified later.
Then every combinatorial line is mapped into a positive homothet of $S$, with scaling $\sum_{i\in I} \lambda_i$ for some non-empty $I\subset [N]$.
Therefore, applying the Hales-Jewett theorem for the pullback of the $k$-colouring of our space gives a monochromatic homothet of $S$.

We still have to specify how we choose the numbers $\lambda_i$.
For $\Psi$ to be injective, we need that $\sum_{i=1}^N \lambda_i(s_{x_i}-s_{x_i'})\ne 0$ if $\underline x\ne \underline x'$.
These can be satisfied for some $1\le \lambda_i\le n^N$ by choosing them sequentially.
This means that the scaling of the obtained monochromatic homothet is at most $\sum_{i=1}^N \lambda_i\le Nn^N$.
\end{proof}

\begin{remark}
The proof above works in every abelian group of sufficiently large cardinality, thus $\Z^d$ in Theorem \ref{gallaiZ} can be replaced with $\R^d$ or any lattice $\mathcal L$.
\end{remark}

\end{document}

%% file: abra4.inc
\begin{center}
\begin{tikzpicture}[every node/.style={inner sep=0pt, minimum size=2.5pt, circle, fill=blue!70!black, draw=black}, line width=0.15pt, draw opacity=.5, scale=3]
\coordinate (P1) at (0.000000, 0.000000);
\coordinate (P2) at (1.000000, 0.000000);
\coordinate (P3) at (0.500000, 0.866025);
\coordinate (P4) at (-0.500000, 0.866025);
\coordinate (P5) at (-1.000000, 0.000000);
\coordinate (P6) at (-0.500000, -0.866025);
\coordinate (P7) at (0.500000, -0.866025);
\coordinate (P8) at (0.728714, 0.684819);
\coordinate (P9) at (0.228714, 0.396143);
\coordinate (P10) at (0.228714, -0.973494);
\coordinate (P11) at (0.271286, 0.684819);
\coordinate (P12) at (0.228714, 0.973494);
\coordinate (P13) at (0.000000, 0.577350);
\coordinate (P14) at (0.228714, -0.396143);
\coordinate (P15) at (0.457427, -0.577350);
\coordinate (P16) at (0.728714, -0.684819);
\coordinate (P17) at (-0.228714, 0.973494);
\coordinate (P18) at (-0.271286, 0.684819);
\coordinate (P19) at (-0.228714, 0.396143);
\coordinate (P20) at (0.500000, -0.288675);
\coordinate (P21) at (0.957427, -0.288675);
\coordinate (P22) at (0.728714, -0.107468);
\coordinate (P23) at (-0.728714, 0.684819);
\coordinate (P24) at (0.457427, -0.000000);
\coordinate (P25) at (0.957427, 0.288675);
\coordinate (P26) at (-0.957427, 0.288675);
\coordinate (P27) at (-0.457427, 0.000000);
\coordinate (P28) at (-0.728714, -0.107468);
\coordinate (P29) at (-0.957427, -0.288675);
\coordinate (P30) at (-0.500000, -0.288675);
\coordinate (P31) at (-0.728714, -0.684819);
\coordinate (P32) at (-0.457427, -0.577350);
\coordinate (P33) at (-0.228714, -0.396143);
\coordinate (P34) at (-0.228714, -0.973494);
\draw (P1) -- (P2);
\draw (P1) -- (P3);
\draw (P1) -- (P4);
\draw (P1) -- (P5);
\draw (P1) -- (P6);
\draw (P1) -- (P7);
\draw (P1) -- (P8);
\draw (P1) -- (P10);
\draw (P1) -- (P12);
\draw (P1) -- (P16);
\draw (P1) -- (P17);
\draw (P1) -- (P21);
\draw (P1) -- (P23);
\draw (P1) -- (P25);
\draw (P1) -- (P26);
\draw (P1) -- (P29);
\draw (P1) -- (P31);
\draw (P1) -- (P34);
\draw (P2) -- (P3);
\draw (P2) -- (P7);
\draw (P2) -- (P11);
\draw (P3) -- (P4);
\draw (P3) -- (P22);
\draw (P4) -- (P5);
\draw (P4) -- (P28);
\draw (P5) -- (P6);
\draw (P5) -- (P18);
\draw (P6) -- (P7);
\draw (P6) -- (P15);
\draw (P7) -- (P32);
\draw (P8) -- (P17);
\draw (P8) -- (P18);
\draw (P8) -- (P19);
\draw (P8) -- (P20);
\draw (P8) -- (P21);
\draw (P9) -- (P15);
\draw (P9) -- (P21);
\draw (P9) -- (P23);
\draw (P9) -- (P30);
\draw (P10) -- (P21);
\draw (P10) -- (P22);
\draw (P10) -- (P24);
\draw (P10) -- (P30);
\draw (P10) -- (P31);
\draw (P11) -- (P20);
\draw (P11) -- (P23);
\draw (P11) -- (P27);
\draw (P12) -- (P23);
\draw (P12) -- (P24);
\draw (P12) -- (P25);
\draw (P13) -- (P14);
\draw (P13) -- (P20);
\draw (P13) -- (P22);
\draw (P13) -- (P25);
\draw (P13) -- (P26);
\draw (P13) -- (P28);
\draw (P13) -- (P30);
\draw (P13) -- (P33);
\draw (P14) -- (P25);
\draw (P14) -- (P28);
\draw (P14) -- (P31);
\draw (P15) -- (P25);
\draw (P15) -- (P30);
\draw (P16) -- (P25);
\draw (P16) -- (P33);
\draw (P16) -- (P34);
\draw (P17) -- (P26);
\draw (P17) -- (P27);
\draw (P18) -- (P24);
\draw (P18) -- (P30);
\draw (P19) -- (P20);
\draw (P19) -- (P29);
\draw (P19) -- (P32);
\draw (P20) -- (P27);
\draw (P20) -- (P30);
\draw (P20) -- (P32);
\draw (P20) -- (P34);
\draw (P22) -- (P33);
\draw (P23) -- (P29);
\draw (P23) -- (P30);
\draw (P24) -- (P30);
\draw (P26) -- (P31);
\draw (P26) -- (P32);
\draw (P26) -- (P33);
\draw (P27) -- (P34);
\draw (P28) -- (P34);
\draw (P29) -- (P34);
\foreach \i in {1,...,34} \node (N\i) at (P\i) {};
\end{tikzpicture}
\end{center}